\documentclass[10pt,a4paper]{amsart}
\pdfoutput=1
\usepackage{amsmath,amsthm,amstext,amssymb,a4,stmaryrd}
\usepackage[T1]{fontenc}

\newcommand{\R}{\ensuremath{\mathbb{R}}}

\newcommand{\del}{{\partial}}

\theoremstyle{plain}
\newtheorem{Thm}{Theorem}[]
\newtheorem{lemma}{Lemma}[]

\theoremstyle{remark}

\theoremstyle{definition}
\newtheorem{Def}{Definition}[]

\begin{document}

\author[]{A. Dirmeier}
\address{Department of Mathematics, Technische Universit{\"a}t Berlin,
Str.~d.~17.~Juni 136, 10623 Berlin, Germany}

\title[]{Causal Classification of Conformally Flat Lorentzian Cylinders}

\begin{abstract}
We derive all possible causality conditions for conformally flat Lorentzian metrics on the two-dimensional cylinder.
\end{abstract}

\maketitle

Let $Z=S^1\times \R$ be the two-dimensional cylinder furnished with natural global coordinates $x\sim x+1\in S^1$ and $y\in\R$. A non-degenerate metric $g$ on $Z$ with signature $(-,+)$ makes $(Z,g)$ a \textbf{Lorentzian cylinder manifold}. Then also $(Z,-g)$ is a Lorentzian cylinder manifold (see e.g~\cite[\S~3.4]{Beem1996}). The Lorentzian metric can be given---with respect to the natural coordinates---as $g=-Edx^2+2Fdxdy+Gdy^2$ with apt functions $E,F,G\colon Z\to\R$.

\begin{Def}
A Lorentzian cylinder manifold $(Z,g)$ such that there is a globally defined timelike vector field $V\colon Z\to TZ$ (i.e.~$g(V,V)<0$ on all of $Z$) will be called a \textbf{Lorentzian cylinder spacetime}.  
\end{Def}

The condition in the Definition above is called time orientability and is necessary do develop causality theory on the Lorentzian spacetime (see e.g.~\cite{Minguzzi2008}).

\begin{Def}
A Lorentzian cylinder spacetime $(Z,g)$ will be called \textbf{conformally flat} if there is a function $\Psi\colon Z\to\R$ and \textit{constants} $E,F,G$ such that the Lorentzian metric takes the form $g=e^{2\Psi}(-Edx^2+2Fdxdy+Gdy^2)$. With $h:=e^{-2\Psi}g$ which is a Lorentzian metric, too, we call the spacetime $(Z,h)$ a \textbf{flat Lorentzian cylinder}.
\end{Def}

As all causality conditions are conformally invariant we can perform all derivations concerning causality of a conformally flat Lorentzian metric by using the conformally transformed flat metric. A flat Lorentzian cylinder admits the following properties.

\begin{lemma}\label{lemm}
Let $(Z,g)$ be a flat Lorentzian cylinder then the following conditions hold:
\begin{itemize}
 \item[(i)] For all $p,q\in Z$ there is an isometry $\Phi\colon Z\to Z$, $\Phi(p)=q$ such that $(\Phi^{\ast}g)|_p=g|_q$.
 \item[(ii)] The Riemannian curvature tensor vanishes identically. 
 \item[(iii)] $(Z,-g)$ is a flat Lorentzian cylinder, too.
\end{itemize}
\end{lemma}

\begin{proof}
(i) As the coefficients $E,F,G$ of the metric are constant it can be easily calculated that $\del_x$ and $\del_y$ are Killing vector fields, i.e.~$L_{\del_x}g=0$ and $L_{\del_y}g=0$. Moving $p$ to $q$ be a suitable combination of the flows of $\del_x$ and $\del_y$ gives the desired isometry. This is a standard result (see e.g.~\cite{Duggal1999}).

(ii) The Riemannian curvature tensor $R$ is given by $R_{abcd}=K(g_{ac}g_{db}-g_{ad}g_{cb})$ with $K$ being the Gaussian curvature. By Brioschi's formula (see e.g.~\cite[pp.~504-507]{Gray1997}) we conclude immediately that $K=0$ if $E,F,G$ are constant and thus $R=0$. 

(iii) Trivially $-g$ is flat, too. Let $V\colon Z\to TZ$ be the timelike vector field assuring the time orientation of $(Z,g)$. Take any non-vanishing vector field $V^{\perp}\colon Z\to TZ$ orthogonal to $V$ with respect to $g$. Then $V^{\perp}$ is everywhere spacelike with respect to $g$ and everywhere timelike with respect to $-g$.
\end{proof}

Item (ii) in the Lemma above justifies the term \textbf{flat} cylinder.

\begin{Thm}
Let $(Z,g)$ be a flat Lorentzian cylinder spacetime. Then the following is equivalent:
\begin{itemize}
 \item[(i)] The fundamental vector $\del_x|_p$ is lightlike with respect to $g$ at some point $p\in Z$. 
 \item[(ii)] The fundamental vector field $\del_x$ is lightlike with respect to $g$ on all of $Z$.
 \item[(iii)] The fundamental vector $\del_x|_p$ is lightlike with respect to $-g$ at some point $p\in Z$. 
 \item[(iv)] The fundamental vector field $\del_x$ is lightlike with respect to $-g$ on all of $Z$.
 \item[(v)] The spacetime $(Z,g)$ is chronological but non-causal.
 \item[(vi)] The spacetime $(Z,-g)$ is chronological but non-causal.
\end{itemize}
\end{Thm}

\begin{proof}
Trivially (ii)$\Rightarrow$(i) and (iv)$\Rightarrow$(iii). By items (i) and (iii) in Lemma \ref{lemm} then also (i)$\Rightarrow$(ii) and (iii)$\Rightarrow$(iv). Furthermore we have $g(\del_x,\del_x)=0=-g(\del_x,\del_x)$ such that (ii)$\Leftrightarrow$(iv). We now prove (ii)$\Leftrightarrow$(v). The reasoning for (iv)$\Leftrightarrow$(vi) is completely analogous. 

"$\Rightarrow$": Assume $\del_x$ is lightlike everywhere with respect to $g$. Then obviously the integral curves of $\del_x$ are closed and lightlike, such that $(Z,g)$ is non-causal. Let $c\colon [0,1]\to Z$ be any timelike curve. Then we have for apt functions $\alpha$ and $\beta$ that $\dot{c}=\alpha\del_x+\beta\del_y$. As $c$ is timelike, $\beta>0$ or $\beta<0$ on all of $c$. But this implies that $c$ can not be closed, hence $(Z,g)$ is chronological.

"$\Leftarrow$": Assume $(Z,g)$ is non-causal. Then there is a curve $c\colon [0,1]\to Z$, which is closed, $c(0)=c(1)$, and lightlike, $g(\dot{c},\dot{c})=0$. Thus there are apt functions $\alpha$ and $\beta$ such that $\dot{c}=\alpha\del_x+\beta\del_y$. But as $c$ is closed there is a $t_0\in[0,1]$ such that $\beta(c(t_0))=0$. This implies there is a point $p:=c(t_0)$ where $\del_x|_p$ is lightlike and hence it is lightlike on all of $Z$.  
\end{proof}

\begin{Thm}
Let $(Z,g)$ be a flat Lorentzian cylinder spacetime. Then the following is equivalent:
\begin{itemize}
 \item[(i)] The fundamental vector $\del_x|_p$ is timelike with respect to $g$ at some point $p\in Z$. 
 \item[(ii)] The fundamental vector field $\del_x$ is timelike with respect to $g$ on all of $Z$. 
 \item[(iii)] The fundamental vector $\del_x|_p$ is spacelike with respect to $-g$ at some point $p\in Z$. 
 \item[(iv)] The fundamental vector field $\del_x$ is spacelike with respect to $-g$ on all of $Z$.
 \item[(v)] The spacetime $(Z,g)$ is totally vicious.
 \item[(vi)] The spacetime $(Z,-g)$ is globally hyperbolic.
\end{itemize}
\end{Thm}

\begin{proof}
Trivially (ii)$\Rightarrow$(i) and (iv)$\Rightarrow$(iii). By items (i) and (iii) in Lemma \ref{lemm} then also (i)$\Rightarrow$(ii) and (iii)$\Rightarrow$(iv). Furthermore we have $g(\del_x,\del_x)<0 \Leftrightarrow -g(\del_x,\del_x)>0$ such that (ii)$\Leftrightarrow$(iv). 

(ii)$\Rightarrow$(v): Assume $\del_x$ is timelike with respect to $g$ then obviously the integral curves of $\del_x$ are closed and timelike. As an integral curve of $\del_x$ passes through every $p\in Z$ the spacetime is totally vicious.

(v)$\Rightarrow$(i): Assume $(Z,g)$ is totally vicious. Then there is a curve $c\colon [0,1]\to Z$, which is closed, $c(0)=c(1)$, and timelike. We choose apt functions $\alpha$ and $\beta$ such that $\dot{c}=\alpha\del_x+\beta\del_y$. As $c$ is closed there is a $t_0\in[0,1]$ such that $\beta(c(t_0))=0$ because otherwise $c$ could not be closed. Hence $\del_x|_p$ is timelike for $p:=c(t_0)$. 

(iv)$\Rightarrow$(vi): By reasoning analogue to above $\del_x$ spacelike, implies $(Z,-g)$ causal, because a closed lightlike curve would have a lightlike tangent vector parallel to $\del_x$ at some point. For two arbitrary points $p=(x_0,y_0)\in Z$ and $q=(x_1,y_1)\in Z$ we will now show that $J^+(p)\cap J^-(q)\subset [y_0,y_1]\times S^1$. Hence the diamonds are compact as subsets of compact sets. We assume without loss of generality that $y_0<y_1$ and the future direction is given by growing $y$. For any causal curve $c\colon [0,1]\to Z$ with $c(0)=p$ and $c\subset J^+(p)$ we must obviously have $\dot{c}=\alpha\del_x+\beta\del_y$ with $\beta >0$ for apt functions $\alpha$ and $\beta$. Hence $J^+(p)\subset [y_0,\infty)\times S^1$ and an analogue reasoning for $q$ implies $J^-(q)\subset (-\infty,y_1]\times S^1$, which proves the assertion.  

(vi)$\Rightarrow$(iv): Assume $(Z,-g)$ is globally hyperbolic. If $\del_x$ was non-spacelike its integral curves would be closed causal curves and hence $(Z,-g)$ would be non-causal.
\end{proof}

These results show that there are only three possible steps for conformally flat Lorentzian cylinder spacetimes on the causal ladder. Either the spacetime is totally vicious or it is chronological but non-causal or it is globally hyperbolic. Furthermore the three steps are complementary with respect to the change of the sign of the metric. Hence for a pair of conformally flat cylinder spacetimes $(Z,\pm g)$ it either holds that both $(Z,g)$ and $(Z,-g)$ are non-causal but chronological or one is totally vicious and the other one is globally hyperbolic.


\begin{thebibliography}{1}
\providecommand{\url}[1]{\texttt{#1}}
\providecommand{\urlprefix}{URL }
\providecommand{\eprint}[2][]{\url{#2}}

\bibitem{Beem1996}
J.~K. Beem, P.~E. Ehrlich, K.~L. Easley, \emph{Global Lorentzian Geometry}
  (Marcel Dekker, Inc., 1996).

\bibitem{Duggal1999}
K.~L. Duggal, R.~Sharma, \emph{Symmetries of Spacetimes and Riemannian
  Manifolds} (Kluwer Academic Publishers, 1999).

\bibitem{Gray1997}
A.~Gray, \emph{Modern Differential Geometry of Curves and Surfaces with
  Mathematica, 2nd ed.} (CRC Press, 1997).

\bibitem{Minguzzi2008}
E.~Minguzzi, M.~Sanchez, \emph{The causal hierarchy of spacetimes}, in:
  \emph{Recent developments in pseudo-Riemannian geometry}, ESI Lect. Math.
  Phys. (Eur. Math. Soc., Z\"urich, 2008), 299--358.

\end{thebibliography}
\end{document}